\documentclass[numbers=enddot,12pt,final,onecolumn,notitlepage]{scrartcl}%
\usepackage[headsepline,footsepline,manualmark]{scrlayer-scrpage}
\usepackage{amssymb}
\usepackage{amsmath}
\usepackage{amsthm}
\usepackage{framed}
\usepackage{comment}
\usepackage{color}
\usepackage[breaklinks=true]{hyperref}
\usepackage[sc]{mathpazo}
\usepackage[T1]{fontenc}
\usepackage[utf8]{inputenc}
\usepackage[russian,english]{babel}
\usepackage{needspace}
\usepackage{tabls}
\providecommand{\U}[1]{\protect\rule{.1in}{.1in}}
\theoremstyle{definition}
\newtheorem{theo}{Theorem}[section]
\newenvironment{theorem}[1][]
{\begin{theo}[#1]\begin{leftbar}}
{\end{leftbar}\end{theo}}
\newtheorem{lem}[theo]{Lemma}
\newenvironment{lemma}[1][]
{\begin{lem}[#1]\begin{leftbar}}
{\end{leftbar}\end{lem}}
\newtheorem{prop}[theo]{Proposition}
\newenvironment{proposition}[1][]
{\begin{prop}[#1]\begin{leftbar}}
{\end{leftbar}\end{prop}}
\newtheorem{defi}[theo]{Definition}

\newtheorem{remk}[theo]{Remark}

\newtheorem{coro}[theo]{Corollary}
\newenvironment{corollary}[1][]
{\begin{coro}[#1]\begin{leftbar}}
{\end{leftbar}\end{coro}}
\newtheorem{conj}[theo]{Conjecture}

\newtheorem{exam}[theo]{Example}
\newenvironment{example}[1][]
{\begin{exam}[#1]\begin{leftbar}}
{\end{leftbar}\end{exam}}
\newtheorem{questn}[theo]{Question}
\newenvironment{quest}[1][]
{\begin{questn}[#1]\begin{leftbar}}
{\end{leftbar}\end{questn}}

\let\sumnonlimits\sum
\let\prodnonlimits\prod
\let\cupnonlimits\bigcup
\let\capnonlimits\bigcap
\renewcommand{\sum}{\sumnonlimits\limits}
\renewcommand{\prod}{\prodnonlimits\limits}
\renewcommand{\bigcup}{\cupnonlimits\limits}
\renewcommand{\bigcap}{\capnonlimits\limits}
\newenvironment{verlong}{}{}
\newenvironment{vershort}{}{}

\excludecomment{verlong}
\includecomment{vershort}
\excludecomment{noncompile}

\ihead{A double Sylvester determinant}
\ohead{page \thepage}
\begin{document}

\title{A double Sylvester determinant}
\author{Darij Grinberg}
\date{
\today
}
\maketitle

\begin{abstract}
\textbf{Abstract.} Given two $\left(  n+1\right)  \times\left(  n+1\right)
$-matrices $A$ and $B$ over a commutative ring, and some $k\in\left\{
0,1,\ldots,n\right\}  $, we consider the $\dbinom{n}{k}\times\dbinom{n}{k}%
$-matrix $W$ whose entries are $\left(  k+1\right)  \times\left(  k+1\right)
$-minors of $A$ multiplied by corresponding $\left(  k+1\right)  \times\left(
k+1\right)  $-minors of $B$. Here we require the minors to use the last row
and the last column (which is why we obtain an $\dbinom{n}{k}\times\dbinom
{n}{k}$-matrix, not a $\dbinom{n+1}{k+1}\times\dbinom{n+1}{k+1}$-matrix). We
prove that the determinant $\det W$ is a multiple of $\det A$ if the $\left(
n+1,n+1\right)  $-th entry of $B$ is $0$. Furthermore, if the $\left(
n+1,n+1\right)  $-th entries of both $A$ and $B$ are $0$, then $\det W$ is a
multiple of $\left(  \det A\right)  \left(  \det B\right)  $. This extends a
previous result of Olver and the author.



\end{abstract}
\tableofcontents

\section{Introduction}

Let $n$ and $k$ be nonnegative integers, and let $A=\left(  a_{i,j}\right)
_{1\leq i\leq n+1,\ 1\leq j\leq n+1}$ be an $\left(  n+1\right)  \times\left(
n+1\right)  $-matrix over some commutative ring. Let $P_{k}$ be the set of all
$k$-element subsets of $\left\{  1,2,\ldots,n\right\}  $. For any such subset
$K\in P_{k}$, let $K+$ denote the subset $K\cup\left\{  n+1\right\}  $ of
$\left\{  1,2,\ldots,n+1\right\}  $. If $U$ and $V$ are two subsets of
$\left\{  1,2,\ldots,n+1\right\}  $, then $\operatorname*{sub}\nolimits_{U}%
^{V}A$ shall denote the $\left\vert U\right\vert \times\left\vert V\right\vert
$-submatrix of $A$ containing only the entries $a_{u,v}$ with $u\in U$ and
$v\in V$. Let $W_{A}$ be the $P_{k}\times P_{k}$-matrix\footnote{This means a
matrix whose rows and columns are indexed by the $k$-element subsets of
$\left\{  1,2,\ldots,n\right\}  $. If you pick a total order on the set
$P_{k}$, then you can view such a matrix as an $\dbinom{n}{k}\times\dbinom
{n}{k}$-matrix.} whose $\left(  I,J\right)  $-th entry (for all $I\in P_{k}$
and $J\in P_{k}$) is
\[
\det\left(  \operatorname*{sub}\nolimits_{I+}^{J+}A\right)  .
\]
(Thus, the entries of $W_{A}$ are all $\left(  k+1\right)  \times\left(
k+1\right)  $-minors of $A$ that use the last row and the last column.) A
particular case of a celebrated result going back to Sylvester \cite{Sylves51}
(see \cite[\S 2.7]{Prasolov} or \cite[Teorema 2.9.1]{Prasol15} or
\cite{Mohr53} for modern proofs) then says that%
\[
\det\left(  W_{A}\right)  =a_{n+1,n+1}^{p}\cdot\left(  \det A\right)
^{q},\qquad\text{where }p=\dbinom{n-1}{k}\text{ and }q=\dbinom{n-1}{k-1}.
\]

Now, consider a second $\left(  n+1\right)  \times\left(  n+1\right)  $-matrix
$B=\left(  b_{i,j}\right)  _{1\leq i\leq n+1,\ 1\leq j\leq n+1}$ over the same
ring. Let $W_{A,B}$ (later to be just called $W$) be the $P_{k}\times P_{k}%
$-matrix whose $\left(  I,J\right)  $-th entry (for all $I\in P_{k}$ and $J\in
P_{k}$) is
\[
\det\left(  \operatorname*{sub}\nolimits_{I+}^{J+}A\right)  \det\left(
\operatorname*{sub}\nolimits_{I+}^{J+}B\right)  .
\]
What can be said about $\det\left(  W_{A,B}\right)  $ ? In general, very
little\footnote{For example, if $n=3$ and $k=2$, then $\det\left(
W_{A,B}\right)  $ is an irreducible polynomial in the (altogether $2\left(
n+1\right)  ^{2}=32$) variables $a_{i,j}$ and $b_{i,j}$ with $110268$
monomials.}. However, under some assumptions, it splits off factors. Namely,
we shall show (Theorem \ref{thm.bisyl.B0}) that $\det\left(  W_{A,B}\right)  $
is a multiple of $\det A$ if $b_{n+1,n+1}=0$. We shall then conclude (Theorem
\ref{thm.bisyl.AB0}) that if both $a_{n+1,n+1}$ and $b_{n+1,n+1}$ are $0$,
then $\det\left(  W_{A,B}\right)  $ is a multiple of $\left(  \det A\right)
\left(  \det B\right)  $. In either case, the quotient (usually a much more
complicated polynomial\footnote{again irreducible in the case when $n=3$ and
$k=2$}) remains mysterious; our proofs are indirect and reveal little about
it. Our second result generalizes a curious property of $\dbinom{n}{2}%
\times\dbinom{n}{2}$-determinants \cite[Theorem 10]{GriOlv18} that arose from
the study of the n-body problem (see Example \ref{exa.k=2} for details).

\subsection*{Acknowledgments}

I thank Christian Krattenthaler, Peter Olver and Victor Reiner for
enlightening discussions, and Peter Olver for the joint work that led to this
paper. The SageMath computer algebra system \cite{SageMath} has been used for
experimentation leading up to some of the results below.

\section{The theorems}

Let us first introduce the standing notations.

Let $\mathbb{N}=\left\{  0,1,2,\ldots\right\}  $. Let $\mathbb{K}$ be a
commutative ring. If $a$ and $b$ are two elements of $\mathbb{K}$, then we
write $a\mid b$ when $b$ is a multiple of $a$ (that is, $b\in\mathbb{K}a$).

If $m\in\mathbb{N}$, then $\left[  m\right]  $ shall mean the set $\left\{
1,2,\ldots,m\right\}  $.

Fix an $n\in\mathbb{N}$. If $K$ is any subset of $\left[  n\right]  $, then
$K+$ shall mean the subset $K\cup\left\{  n+1\right\}  $ of $\left[
n+1\right]  $.

Fix $k\in\left\{  0,1,\ldots,n\right\}  $. Let $P_{k}$ be the set of all
$k$-element subsets of $\left[  n\right]  $. This is a finite set; thus, any
$P_{k}\times P_{k}$-matrix (i.e., any matrix whose rows and columns are
indexed by $k$-element subsets of $\left[  n\right]  $) has a well-defined
determinant\footnote{Here, we are using the concepts of $P \times P$-matrices
(where $P$ is a finite set) and their determinants. Both of these concepts are
folklore; a brief introduction can be found in \cite[\S 1]{Grinbe18}.}. Such
matrices appear frequently in classical determinant theory (see, e.g., the
\textquotedblleft$k$-th compound determinants\textquotedblright\ in
\cite{MuiMet60} and in \cite[\S 2.6]{Prasolov}, as well as the related
\textquotedblleft Generalized Sylvester's identity\textquotedblright\ in
\cite[\S 2.7]{Prasolov} and \cite[Teorema 2.9.1]{Prasol15} and \cite{Mohr53}).

If $A\in\mathbb{K}^{u\times v}$ is a $u\times v$-matrix, and if $I\subseteq
\left[  u\right]  $ and $J\subseteq\left[  v\right]  $, then
$\operatorname*{sub}\nolimits_{I}^{J}A$ shall mean the submatrix of $A$
obtained by removing all rows whose indices are not in $I$ and removing all
columns whose indices are not in $J$. (Rigorously speaking, if $A=\left(
a_{i,j}\right)  _{1\leq i\leq u,\ 1\leq j\leq v}$ and $I=\left\{  i_{1}%
<i_{2}<\cdots<i_{p}\right\}  $ and $J=\left\{  j_{1}<j_{2}<\cdots
<j_{q}\right\}  $, then $\operatorname*{sub}\nolimits_{I}^{J}A=\left(
a_{i_{x},j_{y}}\right)  _{1\leq x\leq p,\ 1\leq y\leq q}$.) When $\left\vert
I\right\vert =\left\vert J\right\vert $, then the submatrix
$\operatorname*{sub}\nolimits_{I}^{J}A$ is square; its determinant
$\det\left(  \operatorname*{sub}\nolimits_{I}^{J}A\right)  $ is called a
\textit{minor} of $A$.

Our main two results are the following:

\begin{theorem}
\label{thm.bisyl.B0}Let $A=\left(  a_{i,j}\right)  _{1\leq i\leq n+1,\ 1\leq
j\leq n+1}\in\mathbb{K}^{\left(  n+1\right)  \times\left(  n+1\right)  }$ and
$B=\left(  b_{i,j}\right)  _{1\leq i\leq n+1,\ 1\leq j\leq n+1}\in
\mathbb{K}^{\left(  n+1\right)  \times\left(  n+1\right)  }$ be such that
$b_{n+1,n+1}=0$. Let $W$ be the $P_{k}\times P_{k}$-matrix whose $\left(
I,J\right)  $-th entry (for all $I\in P_{k}$ and $J\in P_{k}$) is%
\[
\det\left(  \operatorname*{sub}\nolimits_{I+}^{J+}A\right)  \det\left(
\operatorname*{sub}\nolimits_{I+}^{J+}B\right)  .
\]
Then, $\det A\mid\det W$.
\end{theorem}

\begin{theorem}
\label{thm.bisyl.AB0}Let $A=\left(  a_{i,j}\right)  _{1\leq i\leq n+1,\ 1\leq
j\leq n+1}\in\mathbb{K}^{\left(  n+1\right)  \times\left(  n+1\right)  }$ and
$B=\left(  b_{i,j}\right)  _{1\leq i\leq n+1,\ 1\leq j\leq n+1}\in
\mathbb{K}^{\left(  n+1\right)  \times\left(  n+1\right)  }$ be such that
$a_{n+1,n+1}=0$ and $b_{n+1,n+1}=0$. Define the $P_{k}\times P_{k}$-matrix $W$
as in Theorem \ref{thm.bisyl.B0}. Then, $\left(  \det A\right)  \left(  \det
B\right)  \mid\det W$.
\end{theorem}

\begin{example}
\label{exa.k=1}For this example, set $k=1$. Then, $P_{k}=P_{1}=\left\{
\left\{  1\right\}  ,\left\{  2\right\}  ,\ldots,\left\{  n\right\}  \right\}
$. Thus, the map%
\[
\left[  n\right]  \rightarrow P_{k},\ \ \ \ \ \ \ \ \ \ i\mapsto\left\{
i\right\}
\]
is a bijection. Use this bijection to identify the elements $1,2,\ldots,n$ of
$\left[  n\right]  $ with the elements $\left\{  1\right\}  ,\left\{
2\right\}  ,\ldots,\left\{  n\right\}  $ of $P_{k}$. Thus, the $P_{k}\times
P_{k}$-matrix $W$ in Theorem \ref{thm.bisyl.B0} becomes the $n\times n$-matrix%
\begin{align*}
&  \left(  \underbrace{\det\left(  \operatorname*{sub}\nolimits_{\left\{
i\right\}  +}^{\left\{  j\right\}  +}A\right)  }_{=a_{i,j}a_{n+1,n+1}%
-a_{i,n+1}a_{n+1,j}}\ \ \underbrace{\det\left(  \operatorname*{sub}%
\nolimits_{\left\{  i\right\}  +}^{\left\{  j\right\}  +}B\right)  }%
_{=b_{i,j}b_{n+1,n+1}-b_{i,n+1}b_{n+1,j}}\right)  _{1\leq i\leq n,\ 1\leq
j\leq n}\\
&  =\left(  \left(  a_{i,j}a_{n+1,n+1}-a_{i,n+1}a_{n+1,j}\right)  \left(
b_{i,j}\underbrace{b_{n+1,n+1}}_{=0}-b_{i,n+1}b_{n+1,j}\right)  \right)
_{1\leq i\leq n,\ 1\leq j\leq n}\\
&  =\left(  \left(  a_{i,j}a_{n+1,n+1}-a_{i,n+1}a_{n+1,j}\right)  \left(
-b_{i,n+1}b_{n+1,j}\right)  \right)  _{1\leq i\leq n,\ 1\leq j\leq n}.
\end{align*}
This is the matrix obtained from $\left(  \left(  a_{i,j}a_{n+1,n+1}%
-a_{i,n+1}a_{n+1,j}\right)  \right)  _{1\leq i\leq n,\ 1\leq j\leq n}$ by
multiplying the $i$-th row with $-b_{i,n+1}$ for all $i\in\left[  n\right]  $
and multiplying the $j$-th column with $b_{n+1,j}$ for all $j\in\left[
n\right]  $. Thus, the claim of Theorem \ref{thm.bisyl.B0} follows from the
classical fact that%
\[
\det\left(  \left(  a_{i,j}a_{n+1,n+1}-a_{i,n+1}a_{n+1,j}\right)  _{1\leq
i\leq n,\ 1\leq j\leq n}\right)  =a_{n+1,n+1}^{n-1}\cdot\det A.
\]
This fact is known as Chio pivotal condensation (see, e.g., \cite[Theorem
0.1]{KarZha16}), and is a particular case of Sylvester's identity
(\cite[\S 2.7]{Prasolov}).
\end{example}

\begin{example}
\label{exa.k=2}For this example, set $k=2$, and consider the situation of
Theorem \ref{thm.bisyl.B0} again. Then, $P_{k}=P_{2}=\left\{  \left\{
i,j\right\}  \ \mid\ 1\leq i<j\leq n\right\}  $. If $\left\{  i,j\right\}  \in
P_{2}$ and $\left\{  k,l\right\}  \in P_{2}$ satisfy $i<j$ and $k<l$, then the
$\left(  \left\{  i,j\right\}  ,\left\{  k,l\right\}  \right)  $-th entry of
$W$ is%
\begin{align*}
&  \det\underbrace{\left(  \operatorname*{sub}\nolimits_{\left\{  i,j\right\}
+}^{\left\{  k,l\right\}  +}A\right)  }_{=\left(
\begin{array}
[c]{ccc}%
a_{i,k} & a_{i,l} & a_{i,n+1}\\
a_{j,k} & a_{j,l} & a_{j,n+1}\\
a_{n+1,k} & a_{n+1,l} & a_{n+1,n+1}%
\end{array}
\right)  }\ \ \det\underbrace{\left(  \operatorname*{sub}\nolimits_{\left\{
i,j\right\}  +}^{\left\{  k,l\right\}  +}B\right)  }_{\substack{=\left(
\begin{array}
[c]{ccc}%
b_{i,k} & b_{i,l} & b_{i,n+1}\\
b_{j,k} & b_{j,l} & b_{j,n+1}\\
b_{n+1,k} & b_{n+1,l} & 0
\end{array}
\right)  \\\text{(since }b_{n+1,n+1}=0\text{)}}}\\
&  =\det\left(
\begin{array}
[c]{ccc}%
a_{i,k} & a_{i,l} & a_{i,n+1}\\
a_{j,k} & a_{j,l} & a_{j,n+1}\\
a_{n+1,k} & a_{n+1,l} & a_{n+1,n+1}%
\end{array}
\right)  \det\left(
\begin{array}
[c]{ccc}%
b_{i,k} & b_{i,l} & b_{i,n+1}\\
b_{j,k} & b_{j,l} & b_{j,n+1}\\
b_{n+1,k} & b_{n+1,l} & 0
\end{array}
\right)  .
\end{align*}
If we furthermore assume that
\begin{align*}
a_{n+1,n+1}  &  =0,\ \ \ \ \ \ \ \ \ \ \text{and}\\
a_{n+1,i}  &  =a_{i,n+1}=1\ \ \ \ \ \ \ \ \ \ \text{for all }i\in\left[
n\right]  ,\ \ \ \ \ \ \ \ \ \ \text{and}\\
b_{n+1,i}  &  =b_{i,n+1}=1\ \ \ \ \ \ \ \ \ \ \text{for all }i\in\left[
n\right]  ,
\end{align*}
then this entry rewrites as%
\begin{align*}
&  \underbrace{\det\left(
\begin{array}
[c]{ccc}%
a_{i,k} & a_{i,l} & 1\\
a_{j,k} & a_{j,l} & 1\\
1 & 1 & 0
\end{array}
\right)  }_{=a_{j,k}+a_{i,l}-a_{i,k}-a_{j,l}}\underbrace{\det\left(
\begin{array}
[c]{ccc}%
b_{i,k} & b_{i,l} & 1\\
b_{j,k} & b_{j,l} & 1\\
1 & 1 & 0
\end{array}
\right)  }_{=b_{j,k}+b_{i,l}-b_{i,k}-b_{j,l}}\\
&  =\left(  a_{j,k}+a_{i,l}-a_{i,k}-a_{j,l}\right)  \left(  b_{j,k}%
+b_{i,l}-b_{i,k}-b_{j,l}\right)  .
\end{align*}
Hence, \cite[Theorem 10]{GriOlv18} can be obtained from Theorem
\ref{thm.bisyl.AB0} by setting $k=2$ and $A=C_{S}$ and $B=C_{T}$ (and
observing that the matrix $W$ then equals to $W_{S,T}$).
\end{example}

\section{The proofs}

Our proofs of Theorem \ref{thm.bisyl.B0} and Theorem \ref{thm.bisyl.AB0} will
rely on some basic commutative algebra: the notion of a unique factorization
domain (\textquotedblleft UFD\textquotedblright); the concepts of coprime,
prime and irreducible elements; the localization of a commutative ring at a
multiplicative subset. This all appears in most textbooks on abstract algebra;
for example, \cite[Sections VIII.4 and VIII.10]{Knapp16} is a good
reference\footnote{We call \textquotedblleft multiplicative
subset\textquotedblright\ what Knapp (in \cite[Section VIII.10]{Knapp16})
calls a \textquotedblleft multiplicative system\textquotedblright.}.

The \textit{content} of a polynomial $p$ over a UFD is defined to be the
greatest common divisor of the coefficients of $p$. For example, the
polynomial $4x^{2}+6y^{2}\in\mathbb{Z}\left[  x,y\right]  $ has content
$\gcd\left(  4,6\right)  =2$. (Of course, in a general UFD, the greatest
common divisor is defined only up to multiplication by a unit.) The following
known facts are crucial to us:

\begin{proposition}
\label{prop.poly.Zufd}A polynomial ring over $\mathbb{Z}$ in finitely many
indeterminates is always a UFD.
\end{proposition}

\begin{proof}
[Proof of Proposition \ref{prop.poly.Zufd}.]Proposition \ref{prop.poly.Zufd}
appears, e.g., in \cite[Remark after Corollary 8.21]{Knapp16}. For a
constructive proof of Proposition \ref{prop.poly.Zufd}, we refer to
\cite[Chapter IV, Theorems 4.8 and 4.9]{MiRiRu87} or to \cite[Essay 1.4,
Corollary of Theorem 1 and Corollary 1 of Theorem 2]{Edward22}.
\end{proof}

\begin{proposition}
\label{prop.UFD.irred=prime}Let $p$ be an irreducible element of a UFD
$\mathbb{K}$. Then, the quotient ring $\mathbb{K}/\left(  p\right)  $ is an
integral domain.
\end{proposition}

\begin{proof}
[Proof of Proposition \ref{prop.UFD.irred=prime}.]First of all, we recall that
any irreducible element of a UFD is prime (indeed, this follows from
\cite[Proposition 8.13]{Knapp16}). Thus, the element $p$ of $\mathbb{K}$ is
prime. Hence, \cite[Proposition 8.14]{Knapp16} shows that the ideal $\left(
p\right)  $ of $\mathbb{K}$ is prime. Therefore, the quotient ring
$\mathbb{K}/\left(  p\right)  $ is an integral domain. This proves Proposition
\ref{prop.UFD.irred=prime}.
\end{proof}

We shall furthermore use the following properties of contents (whose proofs
are easy):

\begin{proposition}
\label{prop.content.irred}Let $\mathbb{U}$ be a UFD. Let $\mathbb{F}$ be the
field of fractions of $\mathbb{U}$. Let $p\in\mathbb{U}\left[  x_{1}%
,x_{2},\ldots,x_{m}\right]  $ be a polynomial over $\mathbb{U}$. Assume that
the content of $p$ is $1$. Also assume that $p$ is irreducible when considered
as a polynomial in $\mathbb{F}\left[  x_{1},x_{2},\ldots,x_{m}\right]  $.
Then, $p$ is also irreducible when considered as a polynomial in
$\mathbb{U}\left[  x_{1},x_{2},\ldots,x_{m}\right]  $.
\end{proposition}

\begin{proposition}
\label{prop.content.coprime}Let $\mathbb{U}$ be a UFD. Let $p,q\in
\mathbb{U}\left[  x_{1},x_{2},\ldots,x_{m}\right]  $ be two polynomials over
$\mathbb{U}$. Assume that both $p$ and $q$ have content $1$, and assume
furthermore that $p$ and $q$ don't have any indeterminates in common (i.e.,
there is no $i\in\left[  m\right]  $ such that $\deg_{x_{i}}p>0$ and
$\deg_{x_{i}}q>0$). Then, $p$ and $q$ are coprime.
\end{proposition}

The next simple fact states that for any positive integer $n$, the determinant
of the \textquotedblleft generic $n\times n$-matrix\textquotedblright\ (i.e.,
of the $n\times n$-matrix whose $n^{2}$ entries are distinct indeterminates in
a polynomial ring over $\mathbb{Z}$) is irreducible as a polynomial over
$\mathbb{Z}$:

\begin{corollary}
\label{cor.det.irred}Let $n$ be a positive integer. Let $\mathbb{G}$ be the
polynomial ring $\mathbb{Z}\left[  a_{i,j}\ \mid\ \left(  i,j\right)
\in\left[  n\right]  ^{2}\right]  $. Let $\overline{A}\in\mathbb{G}^{n\times
n}$ be the $n\times n$-matrix $\left(  a_{i,j}\right)  _{1\leq i\leq n,\ 1\leq
j\leq n}$. Then, the element $\det\overline{A}$ of $\mathbb{G}$ is irreducible.
\end{corollary}

\begin{proof}
[Proof of Corollary \ref{cor.det.irred}.]A well-known fact (e.g., \cite[Lemma
5.12]{DeKuRo78}) shows that $\det\overline{A}$ is irreducible as an element of
$\mathbb{Q}\left[  a_{i,j}\ \mid\ \left(  i,j\right)  \in\left[  n\right]
^{2}\right]  $. This yields (using Proposition \ref{prop.content.irred}) that
$\det\overline{A}$ is irreducible as an element of $\mathbb{Z}\left[
a_{i,j}\ \mid\ \left(  i,j\right)  \in\left[  n\right]  ^{2}\right]  $ as
well, since the polynomial $\det\overline{A}$ has content $1$. This proves
Corollary \ref{cor.det.irred}.
\end{proof}

An element $a$ of a commutative ring $\mathbb{A}$ is said to be
\textit{regular} if every $b\in\mathbb{A}$ satisfying $ab=0$ must satisfy
$b=0$. (Regular elements are also known as \textit{non-zero-divisors}.) In a
polynomial ring, each indeterminate is regular; hence, each monomial (without
coefficient) is regular (since any product of two regular elements is
regular). The following fact is easy to see:\footnote{We recall a few standard
concepts from commutative algebra:
\par
Let $\mathbb{K}$ be a commutative ring. A \textit{multiplicative subset} of
$\mathbb{K}$ means a subset $S$ of $\mathbb{K}$ that contains the unity
$1_{\mathbb{K}}$ of $\mathbb{K}$ and has the property that every $a,b\in S$
satisfy $ab\in S$.
\par
If $S$ is a multiplicative subset of $\mathbb{K}$, then the
\textit{localization} of $\mathbb{K}$ at $S$ is defined as follows: Let $\sim$
be the binary relation on the set $\mathbb{K}\times S$ defined by%
\[
\left(  \left(  r,s\right)  \sim\left(  r^{\prime},s^{\prime}\right)  \right)
\ \Longleftrightarrow\ \left(  t\left(  rs^{\prime}-sr^{\prime}\right)
=0\text{ for some }t\in S\right)  .
\]
Then, it is easy to see that $\sim$ is an equivalence relation. The set
$\mathbb{L}$ of its equivalence classes $\left[  \left(  r,s\right)  \right]
$ can be equipped with a ring structure via the rules $\left[  \left(
r,s\right)  \right]  +\left[  \left(  r^{\prime},s^{\prime}\right)  \right]
=\left[  \left(  rs^{\prime}+sr^{\prime},ss^{\prime}\right)  \right]  $ and
$\left[  \left(  r,s\right)  \right]  \cdot\left[  \left(  r^{\prime
},s^{\prime}\right)  \right]  =\left[  \left(  rr^{\prime},ss^{\prime}\right)
\right]  $ (with zero element $\left[  \left(  0,1\right)  \right]  $ and
unity $\left[  \left(  1,1\right)  \right]  $). The resulting ring
$\mathbb{L}$ is commutative, and is known as the localization of $\mathbb{K}$
at $S$. (This generalizes the construction of $\mathbb{Q}$ from $\mathbb{Z}$
known from high school.)
\par
The element $\left[  \left(  r,s\right)  \right]  $ of $\mathbb{L}$ is denoted
by $\dfrac{r}{s}$. There is a canonical ring homomorphism from $\mathbb{K}$ to
$\mathbb{L}$ that sends each $r\in\mathbb{K}$ to $\left[  \left(  r,1\right)
\right]  =\dfrac{r}{1}\in\mathbb{L}$.
\par
When all elements of the multiplicative subset $S$ are regular, the statement
\textquotedblleft$t\left(  rs^{\prime}-sr^{\prime}\right)  =0$ for some $t\in
S$\textquotedblright\ in the definition of the relation $\sim$ can be
rewritten in the equivalent (but much simpler) form \textquotedblleft%
$rs^{\prime}=sr^{\prime}$ \textquotedblright\ (which is even more reminiscent
of the construction of $\mathbb{Q}$).}

\begin{proposition}
\label{prop.locali.div}Let $\mathbb{K}$ be a commutative ring. Let $S$ be a
multiplicative subset of $\mathbb{K}$ such that all elements of $S$ are
regular. Let $\mathbb{L}$ be the localization of the ring $\mathbb{K}$ at $S$. Then:

\textbf{(a)} The canonical ring homomorphism from $\mathbb{K}$ to $\mathbb{L}$
is injective. We shall thus consider it as an embedding.

\textbf{(b)} If $\mathbb{K}$ is an integral domain, then $\mathbb{L}$ is an
integral domain.

\textbf{(c)} Let $a$ and $b$ be two elements of $\mathbb{K}$. Then, we have
the following logical equivalence:%
\[
\left(  a\mid b\text{ in }\mathbb{L}\right)  \ \Longleftrightarrow\ \left(
a\mid sb\text{ in }\mathbb{K}\text{ for some }s\in S\right)  .
\]

\end{proposition}

Matrices over arbitrary commutative rings can behave a lot less predictably
than matrices over fields. However, matrices over integral domains still show
a lot of the latter good behavior, such as the following:

\begin{proposition}
\label{prop.intdom.ker-mat}Let $P$ be a finite set. Let $\mathbb{M}$ be an
integral domain. Let $W\in\mathbb{M}^{P\times P}$ be a $P\times P$-matrix over
$\mathbb{M}$. Let $\mathbf{u}\in\mathbb{M}^{P}$ be a vector such that
$\mathbf{u}\neq0$ and $W\mathbf{u}=0$. Here, $\mathbf{u}$ is considered as a
\textquotedblleft column vector\textquotedblright, so that $W\mathbf{u}$ is
defined by
\[
W\mathbf{u}=\left(  \sum_{q\in P}w_{p,q}u_{q}\right)  _{p\in P},\qquad
\text{where }W=\left(  w_{p,q}\right)  _{\left(  p,q\right)  \in P\times
P}\text{ and }\mathbf{u}=\left(  u_{p}\right)  _{p\in P}.
\]

Then, $\det W=0$.
\end{proposition}

\begin{proof}
[Proof of Proposition \ref{prop.intdom.ker-mat}.]Let $m=\left\vert
P\right\vert $. Then, we can view the $P\times P$-matrix $W$ as an $m\times
m$-matrix (by \textquotedblleft numerical reindexing\textquotedblright, as
explained in \cite[\S 1]{Grinbe18}), and we can view the vector $\mathbf{u}$
as a column vector of size $m$. Let us do this from here on.

Let $\mathbb{F}$ be the quotient field of the integral domain $\mathbb{M}$.
Thus, there is a canonical embedding of $\mathbb{M}$ into $\mathbb{F}$. Hence,
we can view the matrix $W\in\mathbb{M}^{m\times m}$ as a matrix over
$\mathbb{F}$, and we can view the vector $\mathbf{u}\in\mathbb{M}^{m}$ as a
vector over $\mathbb{F}$. Let us do so from here on. We are now in the realm
of classical linear algebra over fields: The vector $\mathbf{u}\in
\mathbb{F}^{m}$ is nonzero (since $\mathbf{u}\neq0$) and belongs to the kernel
of the $m\times m$-matrix $W\in\mathbb{F}^{m\times m}$ (since $W\mathbf{u}%
=0$). Hence, the kernel of the matrix $W$ is nontrivial. In other words, this
matrix $W$ is singular. Thus, $\det W=0$ by a classical fact of linear
algebra. This proves Proposition \ref{prop.intdom.ker-mat}.
\end{proof}

Let us next recall an identity for determinants (a version of the
Cauchy--Binet formula):

\begin{lemma}
\label{lem.cauchy-binet}Let $n\in\mathbb{N}$, $m\in\mathbb{N}$ and
$p\in\mathbb{N}$. Let $A\in\mathbb{K}^{n\times p}$ be an $n\times p$-matrix.
Let $B\in\mathbb{K}^{p\times m}$ be a $p\times m$-matrix. Let $k\in\mathbb{N}%
$. Let $P$ be a subset of $\left[  n\right]  $ such that $\left\vert
P\right\vert =k$. Let $Q$ be a subset of $\left[  m\right]  $ such that
$\left\vert Q\right\vert =k$. Then,%
\[
\det\left(  \operatorname*{sub}\nolimits_{P}^{Q}\left(  AB\right)  \right)
=\sum_{\substack{R\subseteq\left[  p\right]  ;\\\left\vert R\right\vert
=k}}\det\left(  \operatorname*{sub}\nolimits_{P}^{R}A\right)  \cdot\det\left(
\operatorname*{sub}\nolimits_{R}^{Q}B\right)  .
\]

\end{lemma}

\begin{proof}
[\nopunct]\noindent Lemma \ref{lem.cauchy-binet} is \cite[Corollary
7.251]{detnotes} (except that we are using the notation $\operatorname*{sub}%
\nolimits_{I}^{J}C$ for what is called $\operatorname*{sub}\nolimits_{w\left(
I\right)  }^{w\left(  J\right)  }C$ in \cite{detnotes}). It also appears in
\cite[Chapter I, (19)]{Gantma00} (where it is stated using $p$-tuples instead
of subsets).
\end{proof}

The next lemma is just a particular case of Theorem \ref{thm.bisyl.B0}, but it
is a helpful stepping stone on the way to proving the latter theorem:

\begin{lemma}
\label{lem.bisyl.AdB0}Let $A=\left(  a_{i,j}\right)  _{1\leq i\leq n+1,\ 1\leq
j\leq n+1}\in\mathbb{K}^{\left(  n+1\right)  \times\left(  n+1\right)  }$ and
$B=\left(  b_{i,j}\right)  _{1\leq i\leq n+1,\ 1\leq j\leq n+1}\in
\mathbb{K}^{\left(  n+1\right)  \times\left(  n+1\right)  }$ be such that
$b_{n+1,n+1}=0$. Assume further that%
\begin{equation}
a_{n+1,j}=0\ \ \ \ \ \ \ \ \ \ \text{for all }j\in\left[  n\right]  .
\label{eq.lem.bisyl.AdB0.ass-j}%
\end{equation}
Define the $P_{k}\times P_{k}$-matrix $W$ as in Theorem \ref{thm.bisyl.B0}.
Then, $\det A\mid\det W$.
\end{lemma}

The following proof is inspired by \cite[proof of Theorem 10]{GriOlv18}.

\begin{proof}
[Proof of Lemma \ref{lem.bisyl.AdB0}.]We WLOG assume that $\mathbb{K}$ is the
polynomial ring over $\mathbb{Z}$ in $n^{2}+\left(  n+1\right)  +\left(
\left(  n+1\right)  ^{2}-1\right)  $ indeterminates%
\begin{align*}
&  a_{i,j}\ \ \ \ \ \ \ \ \ \ \text{for all }i\in\left[  n\right]  \text{ and
}j\in\left[  n\right]  ;\\
&  a_{i,n+1}\ \ \ \ \ \ \ \ \ \ \text{for all }i\in\left[  n+1\right]  ;\\
&  b_{i,j}\ \ \ \ \ \ \ \ \ \ \text{for all }i\in\left[  n+1\right]  \text{
and }j\in\left[  n+1\right]  \text{ except for }b_{n+1,n+1}.
\end{align*}
And, of course, we assume that the entries of $A$ and $B$ that are not zero by
assumption are these indeterminates.\footnote{These assumptions are
legitimate, because if we can prove Lemma \ref{lem.bisyl.AdB0} under these
assumptions, then the universal property of polynomial rings shows that Lemma
\ref{lem.bisyl.AdB0} holds in the general case.}

The ring $\mathbb{K}$ is a UFD (by Proposition \ref{prop.poly.Zufd}).

We WLOG assume that $n>0$ (otherwise, the result follows from $\det
W=\det\left(
\begin{array}
[c]{c}%
0
\end{array}
\right)  =0$).

The set $P_{k}$ is nonempty (since $k\in\left\{  0,1,\ldots,n\right\}  $);
thus, $\left\vert P_{k}\right\vert \geq1$.

Let $\overline{A}$ be the $n\times n$-matrix $\left(  a_{i,j}\right)  _{1\leq
i\leq n,\ 1\leq j\leq n}\in\mathbb{K}^{n\times n}$. Then, because of
(\ref{eq.lem.bisyl.AdB0.ass-j}), we have
\begin{equation}
\det A=a_{n+1,n+1}\cdot\det\overline{A} \label{pf.lem.bisylAdB0.1}%
\end{equation}
(by \cite[Theorem 6.43]{detnotes}, applied to $n+1$ instead of $n$).

The matrix $\overline{A}$ is a completely generic $n\times n$-matrix (i.e.,
its entries are distinct indeterminates); thus, its determinant $\det
\overline{A}$ is an irreducible polynomial in the polynomial ring
$\mathbb{Z}\left[  a_{i,j}\ \mid\ \left(  i,j\right)  \in\left[  n\right]
^{2}\right]  $ (by Corollary \ref{cor.det.irred}). Hence, $\det\overline{A}$
also is an irreducible polynomial in the ring $\mathbb{K}$ (since $\mathbb{K}$
differs from $\mathbb{Z}\left[  a_{i,j}\ \mid\ \left(  i,j\right)  \in\left[
n\right]  ^{2}\right]  $ only in having more variables, which clearly cannot
contribute any factors to $\det\overline{A}$). Thus, Proposition
\ref{prop.UFD.irred=prime} (applied to $p=\det\overline{A}$) shows that the
quotient ring $\mathbb{K}/\left(  \det\overline{A}\right)  $ is an integral domain.

Let $\mathbb{M}$ be the quotient ring $\mathbb{K}/\left(  \det\overline
{A}\right)  $. Then, $\mathbb{M}$ is an integral domain (since $\mathbb{K}%
/\left(  \det\overline{A}\right)  $ is an integral domain). All monomials in
the variables $b_{i,j}$ (with $\left(  i,j\right)  \neq\left(  n+1,n+1\right)
$) are nonzero in $\mathbb{M}$. Likewise, $a_{n+1,n+1}\neq0$ in $\mathbb{M}$.

Let $w$ be the element $\prod_{j\in\left[  n\right]  }b_{n+1,j}\in\mathbb{M}$.
(Strictly speaking, we mean the canonical projection of $\prod_{j\in\left[
n\right]  }b_{n+1,j}\in\mathbb{K}$ onto the quotient ring $\mathbb{M}$.) Then,
$w$ is a nonzero element of the integral domain $\mathbb{M}$ (since
$b_{n+1,j}\neq0$ in $\mathbb{M}$ for all $j\in\left[  n\right]  $).

For each $i\in\left[  n\right]  $, we define $z_{i}\in\mathbb{M}$ by
$z_{i}=\prod_{\substack{j\in\left[  n\right]  ;\\j\neq i}}b_{n+1,j}$
(projected onto $\mathbb{M}$). This is a nonzero element of $\mathbb{M}$. In
$\mathbb{M}$, we have%
\begin{equation}
b_{n+1,i}z_{i}=b_{n+1,i}\prod_{\substack{j\in\left[  n\right]  ;\\j\neq
i}}b_{n+1,j}=\prod_{j\in\left[  n\right]  }b_{n+1,j}=w
\label{pf.lem.bisyl.AdB0.bnizi}%
\end{equation}
for all $i\in\left[  n\right]  $.

We need another piece of notation: If $M$ is a $p\times q$-matrix, and if
$u\in\left[  p\right]  $ and $v\in\left[  q\right]  $, then $M_{\sim u,\sim
v}$ denotes the $\left(  p-1\right)  \times\left(  q-1\right)  $-matrix
obtained from $M$ by removing the $u$-th row and the $v$-th column.

The matrix $A_{\sim1,\sim\left(  n+1\right)  }$ has determinant $0$ (because
(\ref{eq.lem.bisyl.AdB0.ass-j}) shows that its last row consists of zeroes).
In other words, $\det\left(  A_{\sim1,\sim\left(  n+1\right)  }\right)  =0$.

Also, due to (\ref{eq.lem.bisyl.AdB0.ass-j}), we see that each $i\in\left[
n\right]  $ satisfies
\begin{equation}
\det\left(  A_{\sim1,\sim i}\right)  =a_{n+1,n+1}\cdot\det\left(  \overline
{A}_{\sim1,\sim i}\right)  \label{pf.lem.bisyl.AdB0.detA1i}%
\end{equation}
(by \cite[Theorem 6.43]{detnotes}, applied to $A_{\sim1,\sim i}$ instead of
$A$), because the last row of the matrix $A_{\sim1,\sim i}$ is $\left(
0,0,\ldots,0,a_{n+1,n+1}\right)  $.

For each $i\in\left[  n+1\right]  $, we define an element $u_{i}\in\mathbb{M}$
by%
\[
u_{i}=%
\begin{cases}
z_{i}\left(  -1\right)  ^{i}\det\left(  A_{\sim1,\sim i}\right)  , & \text{if
}i\in\left[  n\right]  ;\\
1, & \text{if }i=n+1
\end{cases}
\ \ .
\]
All these $n+1$ elements $u_{1},u_{2},\ldots,u_{n+1}$ of $\mathbb{M}$ are
nonzero\footnote{\textit{Proof.} Each $i\in\left[  n\right]  $ satisfies%
\begin{align*}
u_{i}  &  =z_{i}\left(  -1\right)  ^{i}\underbrace{\det\left(  A_{\sim1,\sim
i}\right)  }_{\substack{=a_{n+1,n+1}\cdot\det\left(  \overline{A}_{\sim1,\sim
i}\right)  \\\text{(by (\ref{pf.lem.bisyl.AdB0.detA1i}))}}}=\underbrace{z_{i}%
}_{\neq0\text{ in }\mathbb{M}}\underbrace{\left(  -1\right)  ^{i}}%
_{\neq0\text{ in }\mathbb{M}}\underbrace{a_{n+1,n+1}}_{\neq0\text{ in
}\mathbb{M}}\cdot\underbrace{\det\left(  \overline{A}_{\sim1,\sim i}\right)
}_{\substack{\neq0\text{ in }\mathbb{M}\\\text{(since }\det\left(
\overline{A}_{\sim1,\sim i}\right)  \text{ is a polynomial}\\\text{of smaller
degree than }\det\overline{A}\text{, and thus}\\\text{is not a multiple of
}\det\overline{A}\text{)}}}\\
&  \neq0\text{ in }\mathbb{M}%
\end{align*}
(since $\mathbb{M}$ is an integral domain). Thus, $u_{1},u_{2},\ldots,u_{n}$
are nonzero. Moreover, $u_{n+1}$ is nonzero (since $u_{n+1}=1$). Thus, we are
done.}.

Let $\mathbf{u}=\left(  u_{J}\right)  _{J\in P_{k}}\in\mathbb{M}^{P_{k}}$ be
the vector defined by%
\[
u_{J}=\prod_{j\in J}u_{j}.
\]
Then, the entries of the vector $\mathbf{u}$ are nonzero (because they are
products of the nonzero elements $u_{1},u_{2},\ldots,u_{n+1}$ of the integral
domain $\mathbb{M}$). Since the vector $\mathbf{u}$ has at least one entry
(because $\left\vert P_{k}\right\vert \geq1$), we thus conclude that
$\mathbf{u}\neq0$.

Let $\Delta$ be the diagonal matrix $\Delta=\operatorname*{diag}\left(
u_{1},u_{2},\ldots,u_{n+1}\right)  \in\mathbb{M}^{\left(  n+1\right)
\times\left(  n+1\right)  }$.

Let $\mathbf{x}\in\mathbb{M}^{n+1}$ be the column vector defined by%
\[
\mathbf{x}=\left(  \left(  -1\right)  ^{1}\det\left(  A_{\sim1,\sim1}\right)
,\left(  -1\right)  ^{2}\det\left(  A_{\sim1,\sim2}\right)  ,\ldots,\left(
-1\right)  ^{n+1}\det\left(  A_{\sim1,\sim\left(  n+1\right)  }\right)
\right)  ^{T}.
\]

Let $\left(  e_{1},e_{2},\ldots,e_{n+1}\right)  $ be the standard basis of the
free $\mathbb{M}$-module $\mathbb{M}^{n+1}$. Thus, for any $\left(
n+1\right)  \times\left(  n+1\right)  $-matrix $C\in\mathbb{M}^{\left(
n+1\right)  \times\left(  n+1\right)  }$ and any $j\in\left\{  1,2,\ldots
,n+1\right\}  $, we have%
\begin{equation}
\left(  \text{the }j\text{-th column of the matrix }C\right)  =Ce_{j}.
\label{pf.lem.bisyl.AdB0.Bej}%
\end{equation}

Now, using Laplace expansion, it is easy to see that%
\begin{equation}
A\mathbf{x}=-\det A\cdot e_{1}. \label{pf.lem.bisyl.AdB0.lapl}%
\end{equation}

\begin{vershort}
[\textit{Proof of (\ref{pf.lem.bisyl.AdB0.lapl}):} Consider the adjugate
$\operatorname*{adj}A$ of the matrix $A$. A standard fact (\cite[Theorem
6.100]{detnotes}) says that $A\cdot\operatorname*{adj}A=\det A\cdot I_{n+1}$.
But the definition of $\operatorname*{adj}A$ reveals that the first column of
the matrix $\operatorname*{adj}A$ is $-\mathbf{x}$. Hence, the first column of
the matrix $A\cdot\operatorname*{adj}A$ is $A\cdot\left(  -\mathbf{x}\right)
=-A\mathbf{x}$. On the other hand, the first column of the matrix
$A\cdot\operatorname*{adj}A$ is $\det A\cdot e_{1}$ (since $A\cdot
\operatorname*{adj}A=\det A\cdot I_{n+1}$).\ Comparing the preceding two
sentences, we conclude that $-A\mathbf{x}=\det A\cdot e_{1}$, so that
$A\mathbf{x}=-\det A\cdot e_{1}$. This proves (\ref{pf.lem.bisyl.AdB0.lapl}).]
\end{vershort}

\begin{verlong}
[\textit{Proof of (\ref{pf.lem.bisyl.AdB0.lapl}):} The quickest way to check
this is to use the adjugate $\operatorname*{adj}A$ of the matrix $A$. A
standard fact (\cite[Theorem 6.100]{detnotes}) says that $A\cdot
\operatorname*{adj}A=\operatorname*{adj}A\cdot A=\det A\cdot I_{n+1}$. But the
definition of $\operatorname*{adj}A$ shows that the $\left(  i,j\right)  $-th
entry of $\operatorname*{adj}A$ is $\left(  -1\right)  ^{i+j}\det\left(
A_{\sim j,\sim i}\right)  $ for each $i\in\left\{  1,2,\ldots,n+1\right\}  $
and each $j\in\left\{  1,2,\ldots,n+1\right\}  $. Hence, the entries in the
first column of the matrix $\operatorname*{adj}A$ are $\left(  -1\right)
^{1+1}\det\left(  A_{\sim1,\sim1}\right)  ,\left(  -1\right)  ^{2+1}%
\det\left(  A_{\sim1,\sim2}\right)  ,\ldots,\left(  -1\right)  ^{\left(
n+1\right)  +1}\det\left(  A_{\sim1,\sim\left(  n+1\right)  }\right)  $ (from
top to bottom). Thus,
\begin{align*}
&  \left(  \text{the first column of the matrix }\operatorname*{adj}A\right)
\\
&  =\left(  \left(  -1\right)  ^{1+1}\det\left(  A_{\sim1,\sim1}\right)
,\left(  -1\right)  ^{2+1}\det\left(  A_{\sim1,\sim2}\right)  ,\ldots,\left(
-1\right)  ^{\left(  n+1\right)  +1}\det\left(  A_{\sim1,\sim\left(
n+1\right)  }\right)  \right)  ^{T}\\
&  =\left(  -\left(  -1\right)  ^{1}\det\left(  A_{\sim1,\sim1}\right)
,-\left(  -1\right)  ^{2}\det\left(  A_{\sim1,\sim2}\right)  ,\ldots,-\left(
-1\right)  ^{n+1}\det\left(  A_{\sim1,\sim\left(  n+1\right)  }\right)
\right)  ^{T}\\
&  \ \ \ \ \ \ \ \ \ \ \left(  \text{since }\left(  -1\right)  ^{i+1}=-\left(
-1\right)  ^{i}\text{ for each }i\in\left\{  1,2,\ldots,n+1\right\}  \right)
\\
&  =-\underbrace{\left(  \left(  -1\right)  ^{1}\det\left(  A_{\sim1,\sim
1}\right)  ,\left(  -1\right)  ^{2}\det\left(  A_{\sim1,\sim2}\right)
,\ldots,\left(  -1\right)  ^{n+1}\det\left(  A_{\sim1,\sim\left(  n+1\right)
}\right)  \right)  }_{\substack{=\mathbf{x}\\\text{(by the definition of
}\mathbf{x}\text{)}}}\\
&  =-\mathbf{x}.
\end{align*}
Hence,%
\[
-\mathbf{x}=\left(  \text{the first column of the matrix }\operatorname*{adj}%
A\right)  =\left(  \operatorname*{adj}A\right)  e_{1}%
\]
(by (\ref{pf.lem.bisyl.AdB0.Bej}), applied to $C=\operatorname*{adj}A$ and
$j=1$). Now,%
\[
A\underbrace{\mathbf{x}}_{=-\left(  -\mathbf{x}\right)  }=A\left(  -\left(
-\mathbf{x}\right)  \right)  =-A\underbrace{\left(  -\mathbf{x}\right)
}_{=\left(  \operatorname*{adj}A\right)  e_{1}}=-\underbrace{A\cdot
\operatorname*{adj}A}_{=\det A\cdot I_{n+1}}\cdot e_{1}=-\det A\cdot e_{1}.
\]
This proves (\ref{pf.lem.bisyl.AdB0.lapl}).]
\end{verlong}

Also, (\ref{pf.lem.bisyl.AdB0.Bej}) (applied to $C=B^{T}$ and $j=n+1$) yields%
\[
B^{T}e_{n+1}=\left(  \text{the }\left(  n+1\right)  \text{-st column of the
matrix }B^{T}\right)  =\left(  b_{n+1,1},b_{n+1,2},\ldots,b_{n+1,n+1}\right)
^{T}.
\]
Hence,%
\begin{align}
\Delta B^{T}e_{n+1}  &  =\Delta\left(  b_{n+1,1},b_{n+1,2},\ldots
,b_{n+1,n+1}\right)  ^{T}\nonumber\\
&  =\left(  u_{1}b_{n+1,1},u_{2}b_{n+1,2},\ldots,u_{n+1}b_{n+1,n+1}\right)
^{T} \label{pf.lem.bisyl.AdB0.DelBTe}%
\end{align}
(since $\Delta=\operatorname*{diag}\left(  u_{1},u_{2},\ldots,u_{n+1}\right)
$).

Now, we claim that%
\begin{equation}
u_{i}b_{n+1,i}=w\cdot\left(  -1\right)  ^{i}\det\left(  A_{\sim1,\sim
i}\right)  \ \ \ \ \ \ \ \ \ \ \text{for each }i\in\left[  n+1\right]  .
\label{pf.lem.bisyl.AdB0.uibi}%
\end{equation}

\begin{vershort}
[\textit{Proof of (\ref{pf.lem.bisyl.AdB0.uibi}):} Let $i\in\left[
n+1\right]  $. If $i=n+1$, then both sides of (\ref{pf.lem.bisyl.AdB0.uibi})
are zero (because $b_{n+1,n+1}=0$ and $\det\left(  A_{\sim1,\sim\left(
n+1\right)  }\right)  =0$). If $i\neq n+1$, then $i\in\left[  n\right]  $ and
thus%
\begin{align*}
\underbrace{u_{i}}_{\substack{=z_{i}\left(  -1\right)  ^{i}\det\left(
A_{\sim1,\sim i}\right)  \\\text{(by the definition of }u_{i}\text{)}%
}}b_{n+1,i}  &  =z_{i}\left(  -1\right)  ^{i}\det\left(  A_{\sim1,\sim
i}\right)  b_{n+1,i}=\underbrace{b_{n+1,i}z_{i}}_{\substack{=w\\\text{(by
(\ref{pf.lem.bisyl.AdB0.bnizi}))}}}\left(  -1\right)  ^{i}\det\left(
A_{\sim1,\sim i}\right) \\
&  =w\cdot\left(  -1\right)  ^{i}\det\left(  A_{\sim1,\sim i}\right)  .
\end{align*}
Hence, (\ref{pf.lem.bisyl.AdB0.uibi}) is proven in both cases.]
\end{vershort}

\begin{verlong}
[\textit{Proof of (\ref{pf.lem.bisyl.AdB0.uibi}):} Let $i\in\left[
n+1\right]  $. We must prove (\ref{pf.lem.bisyl.AdB0.uibi}). If $i=n+1$, then
this is easy (indeed, in this case, both sides are zero, because
$b_{n+1,n+1}=0$ and $\det\left(  A_{\sim1,\sim\left(  n+1\right)  }\right)
=0$). Hence, we WLOG assume that $i\neq n+1$. Hence, $i\in\left[  n\right]  $.
Thus, the definition of $u_{i}$ yields $u_{i}=z_{i}\left(  -1\right)  ^{i}%
\det\left(  A_{\sim1,\sim i}\right)  $. Hence,%
\begin{align*}
u_{i}b_{n+1,i}  &  =z_{i}\left(  -1\right)  ^{i}\det\left(  A_{\sim1,\sim
i}\right)  b_{n+1,i}=\underbrace{b_{n+1,i}z_{i}}_{\substack{=w\\\text{(by
(\ref{pf.lem.bisyl.AdB0.bnizi}))}}}\left(  -1\right)  ^{i}\det\left(
A_{\sim1,\sim i}\right) \\
&  =w\cdot\left(  -1\right)  ^{i}\det\left(  A_{\sim1,\sim i}\right)  .
\end{align*}
This proves (\ref{pf.lem.bisyl.AdB0.uibi}).]
\end{verlong}

Now, (\ref{pf.lem.bisyl.AdB0.DelBTe}) becomes%
\begin{align*}
&  \Delta B^{T}e_{n+1}\\
&  =\left(  u_{1}b_{n+1,1},u_{2}b_{n+1,2},\ldots,u_{n+1}b_{n+1,n+1}\right)
^{T}\\
&  =\left(  w\cdot\left(  -1\right)  ^{1}\det\left(  A_{\sim1,\sim1}\right)
,w\cdot\left(  -1\right)  ^{2}\det\left(  A_{\sim1,\sim2}\right)
,\ldots,w\cdot\left(  -1\right)  ^{n+1}\det\left(  A_{\sim1,\sim\left(
n+1\right)  }\right)  \right)  ^{T}\\
&  \ \ \ \ \ \ \ \ \ \ \left(  \text{by (\ref{pf.lem.bisyl.AdB0.uibi})}\right)
\\
&  =w\cdot\underbrace{\left(  \left(  -1\right)  ^{1}\det\left(  A_{\sim
1,\sim1}\right)  ,\left(  -1\right)  ^{2}\det\left(  A_{\sim1,\sim2}\right)
,\ldots,\left(  -1\right)  ^{n+1}\det\left(  A_{\sim1,\sim\left(  n+1\right)
}\right)  \right)  ^{T}}_{\substack{=\mathbf{x}\\\text{(by the definition of
}\mathbf{x}\text{)}}}\\
&  =w\mathbf{x}.
\end{align*}
Hence,%
\begin{align*}
A\Delta B^{T}e_{n+1}  &  =Aw\mathbf{x}=w\cdot\underbrace{A\mathbf{x}%
}_{\substack{=-\det A\cdot e_{1}\\\text{(by (\ref{pf.lem.bisyl.AdB0.lapl}))}%
}}=-w\cdot\underbrace{\det A}_{\substack{=a_{n+1,n+1}\cdot\det\overline
{A}\\\text{(by (\ref{pf.lem.bisylAdB0.1}))}}}\cdot e_{1}\\
&  =-w\cdot a_{n+1,n+1}\cdot\underbrace{\det\overline{A}}%
_{\substack{=0\\\text{(since we are in }\mathbb{M}\text{)}}}\cdot e_{1}=0.
\end{align*}
In other words, the $\left(  n+1\right)  $-st column of the matrix $A\Delta
B^{T}$ is $0$ (since the $\left(  n+1\right)  $-st column of the matrix
$A\Delta B^{T}$ is $A\Delta B^{T}e_{n+1}$ (by (\ref{pf.lem.bisyl.AdB0.Bej}),
applied to $C=A\Delta B^{T}$ and $j=n+1$)).

Now, fix $I\in P_{k}$. Then, the last column of the matrix
$\operatorname*{sub}\nolimits_{I+}^{I+}\left(  A\Delta B^{T}\right)  $ is $0$
(because this column is a piece of the $\left(  n+1\right)  $-st column of the
matrix $A\Delta B^{T}$, but as we have just shown the latter column is $0$).
Thus, $\det\left(  \operatorname*{sub}\nolimits_{I+}^{I+}\left(  A\Delta
B^{T}\right)  \right)  =0$.

\begin{vershort}
But Lemma \ref{lem.cauchy-binet} (applied to $\mathbb{M}$, $n+1$, $n+1$,
$n+1$, $\Delta B^{T}$, $k+1$, $I+$ and $I+$ instead of $\mathbb{K}$, $n$, $m$,
$p$, $B$, $k$, $P$ and $Q$) yields%
\[
\det\left(  \operatorname*{sub}\nolimits_{I+}^{I+}\left(  A\Delta
B^{T}\right)  \right)  =\sum_{\substack{R\subseteq\left[  n+1\right]
;\\\left\vert R\right\vert =k+1}}\det\left(  \operatorname*{sub}%
\nolimits_{I+}^{R}A\right)  \det\left(  \operatorname*{sub}\nolimits_{R}%
^{I+}\left(  \Delta B^{T}\right)  \right)  .
\]
Comparing this with $\det\left(  \operatorname*{sub}\nolimits_{I+}^{I+}\left(
A\Delta B^{T}\right)  \right)  =0$, we obtain%
\[
0=\sum_{\substack{R\subseteq\left[  n+1\right]  ;\\\left\vert R\right\vert
=k+1}}\det\left(  \operatorname*{sub}\nolimits_{I+}^{R}A\right)  \det\left(
\operatorname*{sub}\nolimits_{R}^{I+}\left(  \Delta B^{T}\right)  \right)  .
\]
In the sum on the right hand side, all addends for which $n+1\notin R$ are
zero (because if $R\subseteq\left[  n+1\right]  $ satisfies $\left\vert
R\right\vert =k+1$ and $n+1\notin R$, then the last row of the matrix
$\operatorname*{sub}\nolimits_{I+}^{R}A$ consists of zeroes\footnote{by
(\ref{eq.lem.bisyl.AdB0.ass-j}), since $n+1\notin R$ but $n+1\in I+$}, and
therefore we have $\det\left(  \operatorname*{sub}\nolimits_{I+}^{R}A\right)
=0$), and thus can be discarded. Hence, we are left with%
\[
0=\sum_{\substack{R\subseteq\left[  n+1\right]  ;\\\left\vert R\right\vert
=k+1;\\n+1\in R}}\det\left(  \operatorname*{sub}\nolimits_{I+}^{R}A\right)
\det\left(  \operatorname*{sub}\nolimits_{R}^{I+}\left(  \Delta B^{T}\right)
\right)  .
\]
But the subsets $R$ of $\left[  n+1\right]  $ satisfying $\left\vert
R\right\vert =k+1$ and $n+1\in R$ can be parametrized as $J+$ with $J$ ranging
over $P_{k}$. Hence, this rewrites further as%
\[
0=\sum_{J\in P_{k}}\det\left(  \operatorname*{sub}\nolimits_{I+}^{J+}A\right)
\det\left(  \operatorname*{sub}\nolimits_{J+}^{I+}\left(  \Delta B^{T}\right)
\right)  .
\]
It is easily seen that $\det\left(  \operatorname*{sub}\nolimits_{J+}%
^{I+}\left(  \Delta B^{T}\right)  \right)  =\det\left(  \operatorname*{sub}%
\nolimits_{I+}^{J+}B\right)  u_{J}$ for each $J\in P_{k}$ (indeed, recall the
definition of $\Delta$ and the fact that $u_{n+1}=1$ and that $\det\left(
C^{T}\right)  =\det C$ for each square matrix $C$). Thus, the above equality
simplifies to
\[
0=\sum_{J\in P_{k}}\det\left(  \operatorname*{sub}\nolimits_{I+}^{J+}A\right)
\det\left(  \operatorname*{sub}\nolimits_{I+}^{J+}B\right)  u_{J}.
\]

\end{vershort}

\begin{verlong}
But $A\Delta B^{T}=A\left(  \Delta B^{T}\right)  $. Hence, Lemma
\ref{lem.cauchy-binet} (applied to $\mathbb{M}$, $n+1$, $n+1$, $n+1$, $\Delta
B^{T}$, $k+1$, $I+$ and $I+$ instead of $\mathbb{K}$, $n$, $m$, $p$, $B$, $k$,
$P$ and $Q$) yields%
\[
\det\left(  \operatorname*{sub}\nolimits_{I+}^{I+}\left(  A\Delta
B^{T}\right)  \right)  =\sum_{\substack{R\subseteq\left[  n+1\right]
;\\\left\vert R\right\vert =k+1}}\det\left(  \operatorname*{sub}%
\nolimits_{I+}^{R}A\right)  \det\left(  \operatorname*{sub}\nolimits_{R}%
^{I+}\left(  \Delta B^{T}\right)  \right)  .
\]
Comparing this with $\det\left(  \operatorname*{sub}\nolimits_{I+}^{I+}\left(
A\Delta B^{T}\right)  \right)  =0$, we obtain%
\begin{align*}
0  &  =\sum_{\substack{R\subseteq\left[  n+1\right]  ;\\\left\vert
R\right\vert =k+1}}\det\left(  \operatorname*{sub}\nolimits_{I+}^{R}A\right)
\det\left(  \operatorname*{sub}\nolimits_{R}^{I+}\left(  \Delta B^{T}\right)
\right) \\
&  =\sum_{\substack{R\subseteq\left[  n+1\right]  ;\\\left\vert R\right\vert
=k+1;\\n+1\in R}}\det\left(  \operatorname*{sub}\nolimits_{I+}^{R}A\right)
\det\left(  \operatorname*{sub}\nolimits_{R}^{I+}\left(  \Delta B^{T}\right)
\right) \\
&  \ \ \ \ \ \ \ \ \ \ +\sum_{\substack{R\subseteq\left[  n+1\right]
;\\\left\vert R\right\vert =k+1;\\n+1\notin R}}\underbrace{\det\left(
\operatorname*{sub}\nolimits_{I+}^{R}A\right)  }_{\substack{=0\\\text{(because
the last row of the}\\\text{matrix }\operatorname*{sub}\nolimits_{I+}%
^{R}A\text{ consists of zeroes}\\\text{(by (\ref{eq.lem.bisyl.AdB0.ass-j}),
since }n+1\notin R\text{, but }n+1\in I+\text{))}}}\det\left(
\operatorname*{sub}\nolimits_{R}^{I+}\left(  \Delta B^{T}\right)  \right) \\
&  =\sum_{\substack{R\subseteq\left[  n+1\right]  ;\\\left\vert R\right\vert
=k+1;\\n+1\in R}}\det\left(  \operatorname*{sub}\nolimits_{I+}^{R}A\right)
\det\left(  \operatorname*{sub}\nolimits_{R}^{I+}\left(  \Delta B^{T}\right)
\right) \\
&  =\sum_{J\in P_{k}}\det\left(  \operatorname*{sub}\nolimits_{I+}%
^{J+}A\right)  \underbrace{\det\left(  \operatorname*{sub}\nolimits_{J+}%
^{I+}\left(  \Delta B^{T}\right)  \right)  }_{\substack{=\left(  \prod_{j\in
J+}u_{j}\right)  \det\left(  \operatorname*{sub}\nolimits_{J+}^{I+}\left(
B^{T}\right)  \right)  \\\text{(since }\Delta\text{ is a diagonal
matrix,}\\\text{and thus }\Delta B^{T}\text{ is just }B^{T}\text{ with
rescaled rows)}}}\\
&  \ \ \ \ \ \ \ \ \ \ \left(
\begin{array}
[c]{c}%
\text{since the subsets }R\text{ of }\left[  n+1\right]  \text{ satisfying
}\left\vert R\right\vert =k+1\\
\text{and }n+1\in R\text{ can be parametrized as }J+\\
\text{with }J\text{ ranging over }P_{k}%
\end{array}
\right) \\
&  =\sum_{J\in P_{k}}\det\left(  \operatorname*{sub}\nolimits_{I+}%
^{J+}A\right)  \underbrace{\left(  \prod_{j\in J+}u_{j}\right)  }%
_{\substack{=\prod_{j\in J}u_{j}\\\text{(since }u_{n+1}=1\text{)}%
}}\underbrace{\det\left(  \operatorname*{sub}\nolimits_{J+}^{I+}\left(
B^{T}\right)  \right)  }_{\substack{=\det\left(  \operatorname*{sub}%
\nolimits_{I+}^{J+}B\right)  \\\text{(since }\det\left(  C^{T}\right)  =\det
C\\\text{for any square matrix }C\text{)}}}
\end{align*}%
\begin{align*}
&  =\sum_{J\in P_{k}}\det\left(  \operatorname*{sub}\nolimits_{I+}%
^{J+}A\right)  \underbrace{\left(  \prod_{j\in J}u_{j}\right)  }%
_{\substack{=u_{J}\\\text{(by the definition of }u_{J}\text{)}}}\det\left(
\operatorname*{sub}\nolimits_{I+}^{J+}B\right) \\
&  =\sum_{J\in P_{k}}\det\left(  \operatorname*{sub}\nolimits_{I+}%
^{J+}A\right)  \det\left(  \operatorname*{sub}\nolimits_{I+}^{J+}B\right)
u_{J}.
\end{align*}

\end{verlong}

Now, forget that we fixed $I$. We thus have proven that%
\begin{equation}
0=\sum_{J\in P_{k}}\det\left(  \operatorname*{sub}\nolimits_{I+}^{J+}A\right)
\det\left(  \operatorname*{sub}\nolimits_{I+}^{J+}B\right)  u_{J}
\label{pf.lem.bisyl.AdB0.Wu=0}%
\end{equation}
for each $I\in P_{k}$. This rewrites as $W\mathbf{u}=0$ (indeed, the left hand
side of (\ref{pf.lem.bisyl.AdB0.Wu=0}) is the $I$-th entry of the zero vector
$0$, whereas the right hand side of (\ref{pf.lem.bisyl.AdB0.Wu=0}) is the
$I$-th entry of $W\mathbf{u}$).

Now, consider the matrix $W$ as a matrix in $\mathbb{M}^{P_{k}\times P_{k}}$.
Then, Proposition \ref{prop.intdom.ker-mat} (applied to $P=P_{k}$) yields
$\det W=0$ in $\mathbb{M}$ (since $\mathbf{u}\neq0$ and $W\mathbf{u}=0$). In
view of the definition of $\mathbb{M}$, this rewrites as $\det\overline{A}%
\mid\det W$ in $\mathbb{K}$.

Let us consider the matrix $W$ again as a matrix over $\mathbb{K}$. Each entry
of $W$ has the form $\det\left(  \operatorname*{sub}\nolimits_{I+}%
^{J+}A\right)  \det\left(  \operatorname*{sub}\nolimits_{I+}^{J+}B\right)  $
for some $I,J\in P_{k}$. Thus, all entries of $W$ are multiples of
$a_{n+1,n+1}$ (since $\det\left(  \operatorname*{sub}\nolimits_{I+}%
^{J+}A\right)  $ is a multiple of $a_{n+1,n+1}$ for all $I,J\in P_{k}%
$\ \ \ \ \footnote{\textit{Proof.} Let $I,J\in P_{k}$. Then, the equality
(\ref{eq.lem.bisyl.AdB0.ass-j}) shows that the last row of the matrix
$\operatorname*{sub}\nolimits_{I+}^{J+}A$ is $\left(  0,0,\ldots
,0,a_{n+1,n+1}\right)  $. Hence, an application of \cite[Theorem
6.43]{detnotes} shows that $\det\left(  \operatorname*{sub}\nolimits_{I+}%
^{J+}A\right)  =a_{n+1,n+1}\det\left(  \operatorname*{sub}\nolimits_{I}%
^{J}A\right)  $. Thus, $\det\left(  \operatorname*{sub}\nolimits_{I+}%
^{J+}A\right)  $ is a multiple of $a_{n+1,n+1}$, qed.}). Hence, the
determinant of $W$ is a multiple of $\left(  a_{n+1,n+1}\right)  ^{\left\vert
P_{k}\right\vert }$, thus a multiple of $a_{n+1,n+1}$ (since $\left\vert
P_{k}\right\vert \geq1$). In other words, $a_{n+1,n+1}\mid\det W$ in
$\mathbb{K}$.

Recall that $\mathbb{K}$ is a UFD. Also, the two polynomials $a_{n+1,n+1}$ and
$\det\overline{A}$ in $\mathbb{K}$ both have content $1$, and don't have any
indeterminates in common; thus, these two polynomials are coprime (by
Proposition \ref{prop.content.coprime}). Hence, any polynomial in $\mathbb{K}$
that is divisible by both $a_{n+1,n+1}$ and $\det\overline{A}$ must be
divisible by the product $a_{n+1,n+1}\cdot\det\overline{A}$ as well. Thus,
from $a_{n+1,n+1}\mid\det W$ and $\det\overline{A}\mid\det W$, we obtain
$a_{n+1,n+1}\cdot\det\overline{A}\mid\det W$. In view of
(\ref{pf.lem.bisylAdB0.1}), this rewrites as $\det A\mid\det W$. This proves
Lemma \ref{lem.bisyl.AdB0}.
\end{proof}

We shall now derive Theorem \ref{thm.bisyl.AB0} from Lemma
\ref{lem.bisyl.AdB0}, following the same idea as in \cite[\S 2.7]{Prasolov}
and \cite[Teorema 2.9.1]{Prasol15} and \cite{Mohr53}:

\begin{proof}
[Proof of Theorem \ref{thm.bisyl.B0}.]We WLOG assume that $n>0$ (otherwise,
the result follows from $\det W=\det\left(
\begin{array}
[c]{c}%
0
\end{array}
\right)  =0$).

We WLOG assume that $\mathbb{K}$ is the polynomial ring over $\mathbb{Z}$ in
$\left(  n+1\right)  ^{2}+\left(  \left(  n+1\right)  ^{2}-1\right)  $
indeterminates%
\begin{align*}
&  a_{i,j}\ \ \ \ \ \ \ \ \ \ \text{for all }i\in\left[  n+1\right]  \text{
and }j\in\left[  n+1\right]  ;\\
&  b_{i,j}\ \ \ \ \ \ \ \ \ \ \text{for all }i\in\left[  n+1\right]  \text{
and }j\in\left[  n+1\right]  \text{ except for }b_{n+1,n+1}.
\end{align*}
And, of course, we assume that the entries of $A$ and $B$ that are not zero by
assumption are these indeterminates. Proposition \ref{prop.poly.Zufd} shows
that the ring $\mathbb{K}$ is a UFD (since it is a polynomial ring over
$\mathbb{Z}$).

Let $S$ be the multiplicative subset $\left\{  a_{n+1,n+1}^{p}\mid
p\in\mathbb{N}\right\}  $ of $\mathbb{K}$. Then, all elements of $S$ are
regular (since they are monomials in a polynomial ring).

Let $\mathbb{L}$ be the localization of the commutative ring $\mathbb{K}$ at
the multiplicative subset $S$. Then, Proposition \ref{prop.locali.div}
\textbf{(a)} shows that the canonical ring homomorphism from $\mathbb{K}$ to
$\mathbb{L}$ is injective; we shall thus consider it as an embedding. Also,
Proposition \ref{prop.locali.div} \textbf{(b)} shows that $\mathbb{L}$ is an
integral domain.

We claim that%
\begin{equation}
\det A\mid\det W\text{ in }\mathbb{L}. \label{pf.thm.bisyl.B0.inL}%
\end{equation}

[\textit{Proof of (\ref{pf.thm.bisyl.B0.inL}):} Consider $A$, $B$ and $W$ as
matrices over $\mathbb{L}$. The entry $a_{n+1,n+1}$ of $A$ is invertible in
$\mathbb{L}$ (by the construction of $\mathbb{L}$). Hence, we can subtract
appropriate scalar multiples\footnote{The scalars, of course, come from
$\mathbb{L}$ here.} of the $\left(  n+1\right)  $-st column of $A$ from each
other column of $A$ to ensure that all entries of the last row of $A$ become
$0$, except for $a_{n+1,n+1}$. (Specifically, for each $j\in\left[  n\right]
$, we have to subtract $a_{n+1,j}/a_{n+1,n+1}$ times the $\left(  n+1\right)
$-st column of $A$ from the $j$-th column of $A$.) All these column
transformations preserve the determinant $\det A$, and also preserve the
minors $\det\left(  \operatorname*{sub}\nolimits_{I+}^{J+}A\right)  $ for all
$I,J\in P_{k}$ (because when the $\left(  n+1\right)  $-st column of $A$ is
subtracted from another column of $A$, the matrix $\operatorname*{sub}%
\nolimits_{I+}^{J+}A$ either stays the same or undergoes an analogous column
transformation\footnote{Here we are using the fact that $n+1\in J+$ (so that
the matrix $\operatorname*{sub}\nolimits_{I+}^{J+}A$ contains part of the
$\left(  n+1\right)  $-st column of $A$).}, which preserves its determinant);
thus, they preserve the matrix $W$. Hence, we can replace $A$ by the result of
these transformations. This new matrix $A$ satisfies
(\ref{eq.lem.bisyl.AdB0.ass-j}). Hence, Lemma \ref{lem.bisyl.AdB0} (applied to
$\mathbb{L}$ instead of $\mathbb{K}$) yields that $\det A\mid\det W$ in
$\mathbb{L}$. This proves (\ref{pf.thm.bisyl.B0.inL}).]

But we must prove that $\det A\mid\det W$ in $\mathbb{K}$. Fortunately, this
is easy: Since $\mathbb{K}$ embeds into $\mathbb{L}$, we can translate our
result \textquotedblleft$\det A\mid\det W$ in $\mathbb{L}$\textquotedblright%
\ as \textquotedblleft$\det A\mid a_{n+1,n+1}^{p}\det W$ in $\mathbb{K}$ for
an appropriate $p\in\mathbb{N}$\textquotedblright\ (by Proposition
\ref{prop.locali.div} \textbf{(c)}, applied to $a=\det A$ and $b=\det W$).
Consider this $p$. The polynomial $a_{n+1,n+1}\in\mathbb{K}$ is coprime to
$\det A$ (this is easily checked\footnote{\textit{Proof.} The polynomial $\det
A$ contains the monomial $a_{1,n+1}a_{2,n}\cdots a_{n+1,1}=\prod_{i\in\left[
n+1\right]  }a_{i,n+2-i}$, and thus is not a multiple of $a_{n+1,n+1}$. Hence,
it is coprime to $a_{n+1,n+1}$ (since the only non-unit divisor of
$a_{n+1,n+1}$ is $a_{n+1,n+1}$ itself, up to scaling by units).}); thus, its
power $a_{n+1,n+1}^{p}$ is coprime to $\det A$ as well. Hence, we can cancel
the $a_{n+1,n+1}^{p}$ from the divisibility $\det A\mid a_{n+1,n+1}^{p}\det
W$, and conclude that $\det A\mid\det W$ in $\mathbb{K}$. This proves Theorem
\ref{thm.bisyl.B0}.
\end{proof}

\begin{proof}
[Proof of Theorem \ref{thm.bisyl.AB0}.]We WLOG assume that $\mathbb{K}$ is the
polynomial ring over $\mathbb{Z}$ in the $\left(  \left(  n+1\right)
^{2}-1\right)  +\left(  \left(  n+1\right)  ^{2}-1\right)  $ indeterminates%
\begin{align*}
&  a_{i,j}\ \ \ \ \ \ \ \ \ \ \text{for all }i\in\left[  n+1\right]  \text{
and }j\in\left[  n+1\right]  \text{ except for }a_{n+1,n+1};\\
&  b_{i,j}\ \ \ \ \ \ \ \ \ \ \text{for all }i\in\left[  n+1\right]  \text{
and }j\in\left[  n+1\right]  \text{ except for }b_{n+1,n+1}.
\end{align*}
And, of course, we assume that the entries of $A$ and $B$ that are not zero by
assumption are these indeterminates. The ring $\mathbb{K}$ is a UFD (by
Proposition \ref{prop.poly.Zufd}).

WLOG assume that $n>0$ (otherwise, the result follows from $\det W=\det\left(
\begin{array}
[c]{c}%
0
\end{array}
\right)  =0$). Thus, the monomial $a_{1,n+1}a_{2,n}\cdots a_{n+1,1}%
=\prod_{i\in\left[  n+1\right]  }a_{i,n+2-i}$ occurs in the polynomial $\det
A$ with coefficient $\pm1$. Hence, the polynomial $\det A$ has content $1$.
Similarly, the polynomial $\det B$ has content $1$.

Theorem \ref{thm.bisyl.B0} yields $\det A\mid\det W$. The same argument yields
$\det B\mid\det W$ (since the matrices $A$ and $B$ play symmetric roles in the
construction of $W$). But Proposition \ref{prop.content.coprime} shows that
the polynomials $\det A$ and $\det B$ in $\mathbb{K}$ are coprime (because
they have content $1$, and don't have any indeterminates in common). Thus, any
polynomial in $\mathbb{K}$ that is divisible by both $\det A$ and $\det B$
must be divisible by the product $\left(  \det A\right)  \left(  \det
B\right)  $ as well. Thus, from $\det A\mid\det W$ and $\det B\mid\det W$, we
obtain $\left(  \det A\right)  \left(  \det B\right)  \mid\det W$. This proves
Theorem \ref{thm.bisyl.AB0}.
\end{proof}

\section{Further questions}

While Theorems \ref{thm.bisyl.B0} and \ref{thm.bisyl.AB0} are now proven, the
field appears far from fully harvested. Three questions readily emerge:

\begin{quest}
What can be said about $\dfrac{\det W}{\det A}$ (in Theorem \ref{thm.bisyl.B0}%
) and $\dfrac{\det W}{\left(  \det A\right)  \left(  \det B\right)  }$ (in
Theorem \ref{thm.bisyl.AB0})? Are there formulas?
\end{quest}

\begin{quest}
Are there more direct proofs of Theorems \ref{thm.bisyl.B0} and
\ref{thm.bisyl.AB0}, avoiding the use of polynomial rings and their properties
and instead \textquotedblleft staying inside $\mathbb{K}$\textquotedblright?
Such proofs might help answer the previous question.
\end{quest}

\begin{quest}
The entries of our matrix $W$ were products of minors of two $\left(
n+1\right)  \times\left(  n+1\right)  $-matrices that each use the last row
and the last column. What can be said about products of minors of two $\left(
n+m\right)  \times\left(  n+m\right)  $-matrices that each use the last $m$
rows and the last $m$ columns, where $m$ is an arbitrary positive integer? The
\textquotedblleft Generalized Sylvester's identity\textquotedblright\ in
\cite[\S 2.7]{Prasolov} answers this for the case of one matrix. It is not
quite obvious what the right analogues of the conditions $a_{n+1,n+1}=0$ and
$b_{n+1,n+1}=0$ are; furthermore, nontrivial examples become even more
computationally challenging.
\end{quest}

\end{document}